\documentclass[12pt]{article}
\usepackage{mathtools,amsthm,amsmath,amsfonts,amssymb,tikz-cd,enumitem, graphicx,mathrsfs,bbm}

\tikzcdset{scale cd/.style={every label/.append style={scale=#1},
    cells={nodes={scale=#1}}}}
\usepackage{authblk}
\usepackage{url}
\usepackage{esvect}
\makeatletter
\newcommand{\address}[1]{\gdef\@address{#1}}
\newcommand{\email}[1]{\gdef\@email{\url{#1}}}
\newcommand{\@endstuff}{\par\vspace{\baselineskip}\noindent\small
\begin{tabular}{@{}l}\@address\\\textit{E-mail address:} \@email\end{tabular}}
\AtEndDocument{\@endstuff}
\makeatother
\usepackage{setspace}
\usepackage[utf8]{inputenc}
\usepackage{CJKutf8}
\counterwithin{equation}{section}
\usepackage[pdfencoding=auto,psdextra,colorlinks=true,linkcolor=blue,citecolor=blue,urlcolor = orange]{hyperref}
\usepackage[margin = 1in]{geometry}
\usepackage{kantlipsum}
\allowdisplaybreaks
\newtheorem{theorem}{Theorem}[section]

\newtheorem{example}[theorem]{Example}

\newtheorem{remark}[theorem]{Remark}
\newtheorem{proposition}[theorem]{Proposition}
\newtheorem{corollary}[theorem]{Corollary}
\newtheorem{assumption}[theorem]{Assumption}
\newcommand{\citep}[1]{\cite{#1}}

\newcommand{\dvol}{\text{\normalfont dvol}}

\newcommand{\even}{\text{\normalfont{even}}}

\newcommand{\professor}{\text{Prof.\ }\hspace{-0.03125mm}}

\begin{document}
\title{\textbf{A vanishing property about the 1-filtered cohomology groups of (4n+2)-dimensional closed symplectic manifolds}}
\author{Hao Zhuang}
\address{Beijing International Center for Mathematical Research, Peking University}
\email{hzhuang@pku.edu.cn}
\date{\today}
\maketitle
\begin{abstract}
This note is a follow-up to our previous work \cite{haozhuangsymplecticsemi1}. For any $(4n+2)$-dimensional closed symplectic manifold, we find that the dimension of the even-degree part of its $1$-filtered cohomology is even, similar to the vanishing property of the classical Euler characteristic of an odd-dimensional closed manifold. We prove our result by constructing and then deforming a skew-adjoint operator. This process follows the methods in \cite{haozhuangsymplecticsemi1} but needs adjustments on signs and the power of the symplectic form. 
\end{abstract} 
\tableofcontents
\section{Introduction}
In the previous work \cite{haozhuangsymplecticsemi1}, we introduced the symplectic semi-characteristic for closed symplectic manifolds. When the dimension of the closed symplectic manifold is $4n$, we obtained a counting formula using a nondegenerate vector field. Due to this counting formula, we may view the symplectic semi-characteristic as an analogue of the classical Euler characteristic. 

However, there is a missing part. For an even dimensional closed manifold, the Poincar\'e-Hopf index formula computes its Euler characteristic. For an odd dimensional closed manifold, its Euler characteristic is $0$. In the symplectic situation, we only have a counting formula for the symplectic semi-characteristic of any $4n$-dimensional closed symplectic manifold, but we do not have a vanishing property in the $(4n+2)$-dimensional case. Actually, for any genus $g$ closed surface (See Remark \ref{counter example 2}), its symplectic semi-characteristic is always $1$, meaning that we cannot expect a vanishing property in the $(4n+2)$-dimensional case. 

In \cite{haozhuangsymplecticsemi1}, we only used the primitive case $(p = 0)$ of Tanaka and Tseng's model \cite[Theorem 1.1]{tanaka_tseng_2018} of Tsai, Tseng, and Yau's $p$-filtered cohomology \cite[(1.2), Theorem 3.1]{tty3rd}. In fact, according to \cite[Proposition 4.8]{tty3rd} or \cite[Theorem 1.3]{tangtsengclausenmappingcone}, for any closed symplectic manifold, the alternating sum of the dimensions of its $p$-filtered cohomology groups is $0$ for any $p\geqslant 0$. If we use the $p\geqslant 1$ filtered cohomology (In Tanaka and Tseng's model, it is revealed by $\wedge\omega^{p+1}$.), we can define another ``semi-characteristic'' for any $(4n+2)$-dimensional closed symplectic manifold, and this new ``semi-characteristic'' may be $0$.

In this note, we show that for any $(4n+2)$-dimensional closed symplectic manifold, the dimension of the even-degree part of its $1$-filtered cohomology is always even. By taking the parity of this dimension, we obtain a vanishing property. 

\begin{assumption}\normalfont
   If there is no particular clarification, $(M,\omega)$ always means a $(4n+2)$-dimensional closed symplectic manifold $M$ admitting a symplectic form $\omega$. 
\end{assumption}

Let $\psi = \omega\wedge\omega$. We present the $1$-filtered case of Tanaka and Tseng's cochain complex (See \cite[Theorem 1.1]{tanaka_tseng_2018} and \cite[(1.1)]{tangtsengclausenmappingcone}). The space of $k$-cochains is $$C^k(M,\psi) \coloneqq \Omega^k(M)\oplus\Omega^{k-3}(M)\ \  (k = 0,1,\cdots,4n+5),$$
where $\Omega^k(M)$ is the space of all smooth forms of degree $k$ on $M$. 
Let $d$ be the exterior derivative on differential forms. Then, we have the coboundary map 
\begin{align*}
\begin{split}
    \partial_{\psi}: C^k(M,\psi)&\to C^{k+1}(M,\psi)
\end{split}
\end{align*}
given by 
\begin{align}\label{mapping cone complex filtered 1}
    \partial_\psi(\alpha,\beta) = (d\alpha_1+\psi\wedge\alpha_2, -d\alpha_2)
\end{align}
for all $(\alpha,\beta)\in\Omega^k(M)\oplus\Omega^{k-3}(M)$. 
The $i$-th cohomology group of $(C^\bullet(M,\psi),\partial_\psi)$ is isomorphic to the $i$-th $1$-filtered cohomology group of $(M,\omega)$. Let $b^\psi_{i}$ be the dimension of the $i$-th cohomology group of $(C^\bullet(M,\psi),\partial_\psi)$. Then, we define the following ``$1$-filtered semi-characteristic''
    \begin{align}\label{symplectic semi definition we need it}
        \ell(M,\psi) = b_{0}^\psi+b_2^\psi+b_4^\psi+\cdots+b_{4n+4}^\psi\mod 2.
    \end{align}
Our main result is as follows. 
\begin{theorem}\label{another main theorem analogous to odd classical diff geo}
    For any $(4n+2)$-dimensional closed symplectic manifold $(M,\omega)$, its ``$1$-filtered semi-characteristic'' $\ell(M,\psi)$ is always $0$. 
\end{theorem}

We prove Theorem \ref{another main theorem analogous to odd classical diff geo} in Section \ref{Doubling the symplectic form section}, following the framework of \cite[Sections 2-4]{haozhuangsymplecticsemi1}. However, compared with the $4n$-dimensional situation \cite{haozhuangsymplecticsemi1}, the current $(4n+2)$-dimensional situation causes differences about plus and minus signs when taking the adjoint of the Clifford action of the volume form, and when swapping the symplectic form with the Clifford action of the volume form. Thus, one technical issue is that the skew-adjoint operator \cite[(2.3)]{haozhuangsymplecticsemi1} cannot be directly imported. We need to use $\psi = \omega\wedge\omega$ and carefully assign plus and minus signs. The usage of $\psi$ also explains why we find the vanishing property in the $1$-filtered situation instead of in the primitive situation. Once the skew-adjoint operator is given, the Witten deformation and the asymptotic analysis follows immediately. In Section \ref{Examples section}, we give some examples and remarks. 

\begin{remark}\normalfont
    At present, we have $\ell(M,\psi) = 0$ for any $(4n+2)$-dimensional symplectic manifold. However, just like \cite{haozhuangsymplecticsemi1} asking for the value of the symplectic semi-characteristic $k(M,\omega)$ for any $(4n+2)$-dimensional symplectic manifold, we can also ask about the value of $\ell(M,\psi)$ when $\dim M = 4n$. We may need more understanding of the primitive cohomology \cite[(3.14), (3.22)]{tty1st} and \cite[(1.5), (1.6)]{tty2nd} and the $p$-filtered cohomology \cite[(1.2), Theorem 3.1]{tty3rd} to give an answer. 
\end{remark}

\vspace{+3mm}
\noindent\textbf{Acknowledgments}. I want to thank my PhD supervisor \professor Xiang Tang and my Postdoc mentor \professor Xiaobo Liu for supporting my research on symplectic invariants. Meanwhile, I want to thank \professor Li-Sheng Tseng for the discussions on the primitive and filtered cohomology groups. Finally, I want to thank Beijing International Center for Mathematical Research for providing a vibrant environment for me to complete this paper. 

\section{Skew-adjoint operator}\label{Doubling the symplectic form section}
We now prove Theorem \ref{another main theorem analogous to odd classical diff geo}. We follow the framework in \cite[Sections 2-4]{haozhuangsymplecticsemi1}, but adapt the construction of the skew-adjoint operator into the $(4n+2)$-dimensional and $1$-filtered case. We provide necessary details for the operator. Afterwards, the asymptotic analysis following the deformation is directly derived from \cite[Sections 3-4]{haozhuangsymplecticsemi1} and thus almost omitted. 

Like in the beginning of \cite[Section 2]{haozhuangsymplecticsemi1}, we assign an almost complex structure $J$ compatible with $\omega$, and then immediately obtain the metric and the $L^2$-norms. Let $d^*$ be the formal adjoint of $d$, and $$\psi^*\lrcorner: \Omega^k(M)\to\Omega^{k-4}(M)$$ be the adjoint of 
\begin{align*}
\psi\wedge: \Omega^k(M)&\to\Omega^{k+4}(M)\\
    \alpha&\mapsto \omega\wedge\omega\wedge\alpha.
\end{align*}
Like in \cite{haozhuangsymplecticsemi1}, we often omit the ``$\lrcorner$'' after $\omega^*$ and the ``$\wedge$'' after $\omega$. Also, we write the coboundary map (\ref{mapping cone complex filtered 1}) into $$\partial_\psi: \begin{bmatrix}
       \alpha_1\\
       \alpha_2
   \end{bmatrix} \mapsto \begin{bmatrix}
    d & \psi\\
    0 & -d
\end{bmatrix}\begin{bmatrix}
    \alpha_1\\
    \alpha_2
\end{bmatrix} = \begin{bmatrix}
    d\alpha_1+\psi\wedge\alpha_2\\
    -d\alpha_2
\end{bmatrix}.$$ The formal adjoint $\partial_\psi^*$ of the boundary map $\partial_\psi$ is 
$$\partial_\psi^* = \begin{bmatrix}
    d^* & 0\\
    \psi^* & -d^*
\end{bmatrix}.$$
Let $C^{\even}(M,\psi)$ be the space $$\sum_{k\text{\ is\ }\even}\left(\Omega^k(M)\oplus\Omega^{k-3}(M)\right).$$
Similar to \cite[Proposition 2.1]{haozhuangsymplecticsemi1}, we have the following Hodge theorem for $$\partial_\psi+\partial_\psi^* = \begin{bmatrix}
        d+d^* & \psi\\
        \psi^* & -d-d^*
    \end{bmatrix}: $$
\begin{proposition}\label{proposition hodge theorem}
     
$\dim\ker\left((\partial_\psi+\partial_\psi^*)\vert_{C^\text{\normalfont even}(M,\psi)}\right) = b^\psi_0+b^\psi_2+\cdots+b^\psi_{4n+4}.$
\end{proposition}

Let $\dvol$ be the volume form of $M$ associated with the metric. Recall the Clifford actions $\hat{c}(\dvol)$ of the volume form presented in \cite[Section 2]{haozhuangsymplecticsemi1}. We adapt it into the current $(4n+2)$-dimensional case and let 
\begin{align*}
    \widetilde{\mathbb{D}} \coloneqq \begin{bmatrix}
        0 & -1\\
        1 & 0
    \end{bmatrix}\begin{bmatrix}
        \hat{c}(\dvol) & \\
        & \hat{c}(\dvol)
    \end{bmatrix}\cdot(\partial_\psi+\partial_\psi^*).
\end{align*}

\begin{proposition}\label{proposition skew-adjoint signature operator}
    The operator $$\widetilde{\mathbb{D}}: C^{\even}(M,\psi)\to C^{\even}(M,\psi)$$ is skew-adjoint. 
\end{proposition}

The proof of Proposition \ref{proposition skew-adjoint signature operator} is similar to the proof of \cite[Proposition 2.4]{haozhuangsymplecticsemi1}. The main differences are the minus signs when swapping $\omega$ with $\hat{c}(\dvol)$ and when taking the adjoint of $\hat{c}(\dvol)$. These differences are caused by $\dim M = 4n+2$. 

\begin{remark}\normalfont
    The construction of $\widetilde{\mathbb{D}}$ is not simply replacing $\omega$ in \cite[(2.3)]{haozhuangsymplecticsemi1} by $\psi$. We also need to adjust the signs in the first matrix. Actually, the minus sign in the first matrix in $\widetilde{\mathbb{D}}$ is also due to the change of the dimension. 
\end{remark}

Recall the concept of Atiyah-Singer mod 2 index (See \cite[Theorem A]{atiyahsingerskewadjoint} and \cite[(7.5)]{wittendeformationweipingzhang}) of any real elliptic skew-adjoint operator. 
It is equal to the parity of the dimension of kernel and is homotopy invariant. By Proposition \ref{proposition hodge theorem}, we have: 
\begin{corollary}\label{corollary signature = semi char}
    $\ell(M,\psi) = \dim\ker\left(\widetilde{\mathbb{D}}\vert_{C^\text{\normalfont even}(M,\psi)}\right)\mod 2$. 
\end{corollary}

Let $\mathbb{D}$ be the skew-adjoint operator $$\begin{bmatrix}
        \dfrac{1}{2}(\psi^*-\psi) & -d-d^*\\
        d+d^* & \dfrac{1}{2}(\psi-\psi^*)
    \end{bmatrix}: C^\even(M,\psi)\to C^\even(M,\psi).$$
Then, we have: 
\begin{proposition}
   $\ell(M,\psi) = \dim\ker\left(\mathbb{D}\vert_{C^\text{\normalfont even}(M,\psi)}\right)\mod 2$. 
\end{proposition}
\begin{proof}
    We find that the operator $\widetilde{\mathbb{D}}$ is equal to
    \begin{align*}
        \begin{bmatrix}
        -1 & 0 \\
         0 & 1
    \end{bmatrix}\begin{bmatrix}
        \hat{c}(\dvol) & \\
         & \hat{c}(\dvol)
    \end{bmatrix}\hspace{-1mm}\cdot\hspace{-0.5mm}\mathbb{D} + \dfrac{1}{2}\begin{bmatrix}
        -\hat{c}(\dvol)(\psi^*+\psi) & \\
        & \hat{c}(\dvol)(\psi+\psi^*)
    \end{bmatrix}.
    \end{align*}
    The proposition is then guaranteed by Corollary \ref{corollary signature = semi char} and the homotopy invariance of the Atiyah-Singer mod 2 index.
\end{proof}

 Let $V$ be a nondegenerate smooth vector field on $M$. For the definition of nondegenerate vector fields, see \cite[Section 1.6]{bgv} and \cite[Definition 1.4]{haozhuangsymplecticsemi1}. Similar to \cite[Section 3]{haozhuangsymplecticsemi1}, for any $T>0$, we have the following Witten deformation 
\begin{align}\label{witten deformation finally appears}
    \mathbb{D}_{T}\coloneqq \begin{bmatrix}
    \dfrac{1}{2}(\psi^*-\psi) & -d-d^*-T\hat{c}(V)\\ 
    d+d^*+T\hat{c}(V) & \dfrac{1}{2}(\psi-\psi^*)
\end{bmatrix}
\end{align}
of the operator $\mathbb{D}$. By the homotopy invariance of the Atiyah-Singer mod 2 index, 
\begin{align*}
    \ell(M,\psi) = & \dim\ker\left(\mathbb{D}\vert_{C^\text{\normalfont even}(M,\psi)}\right)\mod 2 \\
    = & \dim\ker\left(\mathbb{D}_T\vert_{C^\text{\normalfont even}(M,\psi)}\right)\mod 2.
\end{align*}
Now, we conduct all the asymptotic analysis in \cite[Sections 3-4]{haozhuangsymplecticsemi1} on the operator $\mathbb{D}_T$ but with two adjustments: 
\begin{enumerate}[label = (\arabic*)]
    \item Replacing the dimension $4n$ by $4n+2$;
    \item Replacing the form $\omega$ by $\psi$. 
\end{enumerate}
For more about the asymptotic analysis that we use, see \cite[Chapters VIII-X]{bismutandlebeau}, \cite[Section 2.2]{zhangcountingmod2indexkervairesemi}, \cite[Chapters 4-7]{wittendeformationweipingzhang}, and \cite[Chapters 4-7]{weipingzhangnewedition}. After that, we find
\begin{align}\label{second to last conclusion}
    \dim\ker\left(\mathbb{D}_T\vert_{C^\text{even}(M,\psi)}\right) = \text{the number of zero points of\ }V.
\end{align}
 Since $\dim M = 4n+2$, a standard fact \cite[Theorem 2.6]{zuoqinwangnote27} is that the parity of the right hand side of (\ref{second to last conclusion}) is always even. Therefore, 
 \begin{align*}
     & \hspace{+1mm}\ell(M,\psi)\\
     = &\hspace{+0.6mm} \dim\ker\left(\mathbb{D}_T\vert_{C^\text{even}(M,\psi)}\right)\mod 2 \\
     = &\ \text{the number of zero points of\ }V\mod 2 \\
     = &\ 0.
 \end{align*}
 The proof of Theorem \ref{another main theorem analogous to odd classical diff geo} is complete.

\section{Examples and remarks}\label{Examples section}
In this section, we use both \cite[(3.2)]{tangtsengclausenmappingcone} and our Theorem \ref{another main theorem analogous to odd classical diff geo} to obtain $\ell(M,\psi) = 0$ for two symplectic manifolds. Also, we give two remarks about $\ell(M,\psi)$ and the symplectic semi-characteristic $k(M,\omega)$. 

We first recall the formula \cite[(3.2)]{tangtsengclausenmappingcone} of $b_i^\psi$. Let $H_{\text{dR}}^i(M)$ be the $i$-th de Rham cohomology group of $M$, $b_i$ be the dimension of $H_{\text{dR}}^i(M)$, and $r_i$ be the rank of the map
\begin{align}\label{the psi map}
\begin{split}
    \psi\wedge: H_{\text{dR}}^i(M) &\to H_{\text{dR}}^{i+4}(M)\\
    \alpha &\mapsto \psi\wedge\alpha = \omega\wedge\omega\wedge\alpha.
\end{split}
\end{align}
Then, we have 
\begin{align}\label{the formula for b i psi}
    b_i^\psi = b_i-r_{i-4}+b_{i-3}-r_{i-3}.
\end{align}
Let $\mathbb{S}^2$ be the $2$-dimensional unit sphere, $\omega_{\mathbb{S}^2}$ be the standard symplectic form on $\mathbb{S}^2$, and $h: \mathbb{S}^2\to\mathbb{R}$ be the height function \cite[Example 3.4]{banyaga2013lectures}.
\begin{example}\normalfont \label{weak replacement of the euler char}
\normalfont
    Let $M$ be $\mathbb{S}^2\times\mathbb{S}^2\times\mathbb{S}^2$. Let $\omega_j$ ($j = 1, 2, 3$) be the pullback of $\omega_{\mathbb{S}^2}$ onto $\mathbb{S}^2\times\mathbb{S}^2\times\mathbb{S}^2$ via the $j$-th projection. The symplectic form on $\mathbb{S}^2\times\mathbb{S}^2\times\mathbb{S}^2$ is then 
    $\omega_1+\omega_2+\omega_3.$ The function
    \begin{align*}
    f: \mathbb{S}^2\times\mathbb{S}^2\times\mathbb{S}^2&\to\mathbb{R}\\
    (q_1,q_2,q_3)&\mapsto h(q_1)+h(q_2)+h(q_3)
    \end{align*}
     is a Morse function with $8$ critical points. By Theorem \ref{another main theorem analogous to odd classical diff geo},  
    $k(\mathbb{S}^2\times\mathbb{S}^2\times\mathbb{S}^2,\psi) = 0$.

    Alternatively, we can use \cite[(3.2)]{tangtsengclausenmappingcone}. Here, $\psi$ is equal to 
    $$2\omega_1\wedge\omega_2+2\omega_1\wedge\omega_3+2\omega_2\wedge\omega_3.$$
    Using the K\"unneth formula \cite[Section 5]{bott1982differential}, we find
    $$b_0 = b_6 = 1, b_1 = b_3 = b_5 = 0, b_2 = b_4 = 3.$$
    Then, using the basis of the de Rham cohomology of $\mathbb{S}^2\times\mathbb{S}^2\times\mathbb{S}^2$, we find
    $$r_0 = r_2 = 1, r_1 = r_3 = r_4 = r_5 = r_6 = 0.$$
    For example, $H_{\text{dR}}^2(\mathbb{S}^2\times\mathbb{S}^2\times\mathbb{S}^2)$ has a basis
    $$\omega_1, \omega_2, \omega_3.$$
    The map (\ref{the psi map}) maps this basis to only one element
    $$2\omega_1\wedge\omega_2\wedge\omega_3$$
    in $H^6_{\text{dR}}(\mathbb{S}^2\times\mathbb{S}^2\times\mathbb{S}^2),$
    so the rank $r_2 = 1$. Other $r_i$'s are computed similarly. By (\ref{the formula for b i psi}), $$b_0^\psi+b_2^\psi+b_4^\psi+b_6^\psi+b_8^\psi = 1+3+2+0+0 = 6, $$
    showing that $\ell(\mathbb{S}^2\times\mathbb{S}^2\times\mathbb{S}^2,\psi) = 0$.
\end{example}

\begin{example}\normalfont
    Recall the Kodaira-Thurston four-fold given in \cite[Section 3.4]{tty1st} and \cite[(5.3)]{tanaka_tseng_2018}: By identifying points
   $$(x_1,x_2,x_3,x_4)\sim (x_1 + a, x_2+b, x_3+c, x_4+d - bx_3) \text{\ \ (when $a,b,c,d\in\mathbb{Z}$)}$$
    on $\mathbb{R}^4$, the Kodaira-Thurston four-fold is $K\coloneqq \mathbb{R}^4/\sim$. 
    It is a non-K\"ahler closed symplectic manifold equipped with the symplectic form
    \begin{align*}
       \omega_{K} = dx_1\wedge dx_2 + dx_3\wedge(dx_4+x_2dx_3).
    \end{align*}
    Now, we let $M = K\times\mathbb{S}^2$ and choose the symplectic form $\omega = \omega_K + \omega_{\mathbb{S}^2}$ on $M$ (Here, we omit the pullbacks by projections.). Let $\nabla h$ be the gradient vector field of the height function $h$ on $\mathbb{S}^2$ mentioned in Example \ref{weak replacement of the euler char}. Then, the vector field $$\dfrac{\partial}{\partial x_1}+\nabla h$$ is nondegenerate and has $2$ zero points. By Theorem \ref{another main theorem analogous to odd classical diff geo}, $k(M,\psi) = 0$. 

    On the other hand, using the de Rham cohomology \cite[Figure 5]{tanaka_tseng_2018} of $K$ together with the K\"unneth formula \cite[Section 5]{bott1982differential}, we find that for $M = K\times\mathbb{S}^2,$ 
    \begin{align*}
        b_0 = b_6 = 1, b_1 = b_5 = 3, b_2 = b_4 = 5, b_3 = 6,\\ 
    r_0 = r_2 = 1, r_1 = 2, r_3 = r_4 = r_5 = r_6 = 0. 
    \end{align*}
    Thus, by (\ref{the formula for b i psi}), for $M = K\times\mathbb{S}^2,$ 
    $$b_0^\psi+b_2^\psi+b_4^\psi+b_6^\psi+b_8^\psi = 1+5+5+6+3 = 20, $$
    meaning that $\ell(M,\psi) = 0$. 
\end{example}

\begin{remark}\label{counter example 2}
\normalfont
 We emphasize again that Theorem \ref{another main theorem analogous to odd classical diff geo} is not a solution to the $(4n+2)$-dimensional case of \cite[Question 1.1]{haozhuangsymplecticsemi1}. In fact, by \cite[(4.4)]{tangtsengclausensymplecticwitten}, for any genus $g$ closed surface $\Sigma_g$, the dimension of the even-degree part of its primitive cohomology is $1+2g$. Therefore, its symplectic semi-characteristic $k(\Sigma_g,\omega)$ is $1$, not the same as $\ell(\Sigma_g,\psi) = 0$. 
\end{remark}

\begin{remark}\normalfont
    Another interesting phenomenon is that the analogue of the classical Euler characteristic separates into the versions of ``the primitive cohomology for $\dim M = 4n$'' and ``the $1$-filtered cohomology for $\dim M = 4n+2$'' respectively. The mechanism behind this phenomenon is not clear, and we hope to find a way to unify the two versions. 
\end{remark}

\bibliographystyle{abbrv}
\bibliography{mybib.bib}
\end{document}